\documentclass[]{amsart}
\calclayout
\usepackage{amsmath,amssymb,amsthm}
\usepackage{enumitem}
\usepackage{fixltx2e}
\usepackage{mathtools}
\usepackage{tikz}
\usepackage{tikz-cd}
\usepackage{xparse}
\usepackage{xspace}
\usepackage{mathtools}
\usepackage{dsfont}
\usepackage[all]{xy}
\usepackage{color}
\usepackage{verbatim}
\usepackage{hyperref}
\usepackage{mathrsfs}
\usepackage{microtype}
\mathtoolsset{showonlyrefs}
\usepackage{bm}
\usepackage{lipsum}

\usepackage{scalerel}
\usepackage{stackengine,wasysym}

\numberwithin{equation}{section}
\mathtoolsset{showonlyrefs=true}

\let\cal\mathcal

\def\Ascr{{\cal A}}
\def\Bscr{{\cal B}}
\def\Cscr{{\cal C}}
\def\Dscr{{\cal D}}

\def\Fscr{{\cal F}}

\def\Oscr{{\cal O}}

\def\Tscr{{\cal T}}
\def\Uscr{{\cal U}}
\def\Vscr{{\cal V}}

\let\blb\mathbb
\def\CC{{\blb C}}

\def \ZZ{{\blb Z}}

\def \NN{{\blb N}}

\def\Dis{\operatorname{Dis}}

\def\id{\text{id}}
\def\Id{\operatorname{id}}

\def\Res{\operatorname{Res}}

\def\Lotimes{\overset{L}{\otimes}}

\def\mod{\operatorname{mod}}

\def\rad{\operatorname {rad}}
\def\PC{\operatorname {PC}}
\def\dis{\operatorname {dis}}
\def\pc{\operatorname {pc}}

\def\Rep{\operatorname {rep}}
\def\GL{\operatorname {GL}}

\def\Ext{\operatorname {Ext}}

\def\End{\operatorname {End}}

\def\Gl{\operatorname {GL}}

\def\im{\operatorname {im}}

\def\coker{\operatorname {coker}}
\def\ker{\operatorname {ker}}
\def\Ker{\operatorname {ker}}

\def\End{\operatorname {End}}

\def\id{{\operatorname {id}}}

\def\r{\rightarrow}
\def\d{\downarrow}
\def\u{\uparrow}

\def\GL{\operatorname {GL}}
\def\cohom{\operatorname{cohom}}

\DeclareMathOperator{\Ind}{Ind}

\DeclareMathOperator{\Aut}{Aut}

\DeclareMathOperator{\eval}{ev}

\DeclareMathOperator{\Ob}{Ob}
\DeclareMathOperator{\Sym}{Sym}

\DeclareMathOperator{\SL}{SL}

\DeclareMathOperator{\addc}{add}
\DeclareMathOperator{\perf}{Perf}
\DeclareMathOperator{\Ab}{Ab}

\DeclareMathOperator{\usl}{\underline{sl}}

\theoremstyle{definition}

\newtheorem{lemma}{Lemma}[section]
\newtheorem{proposition}[lemma]{Proposition}
\newtheorem{theorem}[lemma]{Theorem}
\newtheorem{corollary}[lemma]{Corollary}

\newtheorem{example}[lemma]{Example}
\newtheorem{definition}[lemma]{Definition}

\newtheorem{remark}[lemma]{Remark}

\DeclareMathOperator\derived{\mathbf{D}}

\DeclareMathOperator\Hom{Hom}

\DeclareMathOperator\comod{comod}

\DeclareMathOperator\Vect{vect}
\DeclareMathOperator\VVect{Vect}
\DeclareMathOperator\determinant{det}
\DeclareMathOperator\coend{coend}
\DeclareMathOperator\aut{\underline{aut}}

\DeclareMathOperator\induced{ind}
\DeclareMathOperator{\uaut}{\underline{gl}}
\DeclareMathOperator{\uend}{\underline{end}}

\DeclareMathOperator{\dist}{dis}

\def\perf{\operatorname{perf}}

\def\opp{\operatorname{op}}

\def\cha{\operatorname{char}}
\makeatletter
\newcommand*\bigcdot{\mathpalette\bigcdot@{.5}}
\newcommand*\bigcdot@[2]{\mathbin{\vcenter{\hbox{\scalebox{#2}{$\m@th#1\bullet$}}}}}
\makeatother

\usepackage{xcolor}
\hypersetup{
    colorlinks,
    linkcolor={red!50!black},
    citecolor={blue!50!black},
    urlcolor={blue!80!black}
}

\mathchardef\mhyphen="2D

\makeatletter
\newcommand{\mylabel}[2]{#2\def\@currentlabel{#2}\label{#1}}
\makeatother

\title[The Tannaka-Krein formalism \ldots]{The Tannaka-Krein formalism and (re)presentations of universal quantum groups}
\author{Theo Raedschelders}
\email[Theo Raedschelders]{Theo.Raedschelders@glasgow.ac.uk}
\address{School of Mathematics and Statistics, University of Glasgow, Glasgow G12 8QQ, 
United Kingdom}
\thanks{The first author is supported by EPSRC postdoctoral fellowship EP/R005214/1, and would like to thank Julian K\"ulshammer and Dimitry Leites for valuable comments.}
\author{Michel Van den Bergh}
\email[Michel Van den Bergh]{michel.vandenbergh@uhasselt.be}
\address{Departement WNI, Universiteit Hasselt, Universitaire Campus \\
B-3590 Diepenbeek}
\thanks{The second author is a senior researcher at the Research Foundation Flanders (FWO).  While working on this project he was supported by
the FWO grant G0D8616N: ``Hochschild cohomology and deformation theory of triangulated categories''.}

\let\oldmarginpar\marginpar
\def\marginpar#1{\oldmarginpar{\tiny \raggedright #1}}

\begin{document}

\begin{abstract}
This is a draft version for an extra chapter in the second edition of the book ``Quantum Groups and Noncommutative Geometry" by Yu. I. Manin \cite{manin-second}. We survey our work in \cite{thesis,MR3570144,MR3713007}, placing particular emphasis on the Tannaka-Krein formalism.
\end{abstract}

\maketitle

\tableofcontents

\section{Introduction}
In \cite[\S 13.8]{MR1016381}, Manin discusses the possibility of ``hidden symmetry" in algebraic geometry. He showed that certain universal quantum groups coact on the homogeneous coordinate ring of an (embedded) projective variety, and these quantum groups are typically much larger than the honest automorphism groups of the variety. In fact, Manin's construction works much more generally, and in this survey, we aim to convince the reader that, as long as one starts with a reasonable algebra $A$, the universal bi- and Hopf algebras $\uend(A)$ and $\uaut(A)$ coacting on $A$ are well behaved objects. 

To do so, it seems natural to start by looking at $A=\mathbb{K}[x_1,\ldots,x_n]$, which we think of as a homogeneous coordinate ring for projective space. Our work in \cite{thesis,MR3570144,MR3713007} has focused on the representation theory of the universal Hopf algebra $\uaut(A)$ coacting on $A$ (and on its noncommutative counterparts). We show that the representations are as nice as can be: the category of comodules for $\uaut(A)$ can be given the structure of a highest weight category (see \ref{def:qh}), and it shares many more similarities with the category of rational representations of the general linear group $\GL_n$, or equivalently, the category of comodules over the coordinate Hopf algebra $\Oscr(\GL_n)$. In contrast to $\Oscr(\GL_n)$, however, the universal Hopf algebras $\uaut(A)$ have
rather complicated presentations and are, moreover, of exponential growth. In order to deal with them, we resort to using a different set of techniques, which go by the name of the \textit{Tannaka-Krein formalism}\index{Tannaka-Krein formalism}. 

\section{The Tannaka-Krein formalism}
\label{sec:tk}
\subsection{The basic example}

Consider a finite group $G$. The starting point for the Tannaka-Krein formalism is the basic question: can $G$ be recovered from its category of finite-dimensional complex representations $\Rep_{\CC}(G)$?

As stated, this question is at best unclear, since one needs to specify what structure on the category $\Rep_{\CC}(G)$ is taken into account. Indeed, one can consider $\Rep_{\CC}(G)$ as:
\begin{enumerate}
\item (abelian) category,
\item monoidal category,
\item symmetric monoidal category.
\end{enumerate}

For $(1)$, the answer is no: for any finite group $H$ with the same number of conjugacy classes as $G$, say $n$, there is an equivalence of categories $\Rep_{\CC}(G) \cong \Rep_{\CC}(H)$, since by the Artin-Wedderburn theorem both categories are equivalent to the category $\mod(\CC^n)$.

For $(2)$, the answer is also no, though this is more subtle.
Consider, for example, the two non-abelian groups of order $8$, the dihedral group $D_8$ and the quaternion group $Q_8$. These groups have the same character table, and even isomorphic Grothendieck rings, but one can check that $\Rep_{\CC}(D_8)$ and $\Rep_{\CC}(Q_8)$ are not equivalent as monoidal categories. Nevertheless, there are two non-isomorphic groups $G$ and $H$ of order $64$, both of which arise as semi-direct products of $\ZZ/2\ZZ \times \ZZ/2\ZZ$ and $\ZZ/4\ZZ \times \ZZ/4\ZZ$, such that $\Rep_{\CC}(G) \cong \Rep_{\CC}(H)$ as monoidal categories, see \cite{MR1932664}. Finite groups with monoidally equivalent categories of finite-dimensional complex representations are called \textit{isocategorical}\index{isocategorical} in \cite{MR1810480}. In loc. cit. all groups isocategorical to a given group are classified in terms of group-theoretical data.

For $(3)$, the answer is yes: by \cite[Theorem 3.2 (b)]{MR654325}, the forgetful functor 
\begin{equation}
\label{eq:forgetful}
F:\Rep_{\CC}(G) \to \Vect_{\CC}
\end{equation}
is the unique $\CC$-linear exact faithful symmetric monoidal functor from $\Rep_{\CC}(G)$ to $\Vect_{\CC}$, where $\Vect_{\CC}$ is the category of finite-dimensional $\CC$-vector spaces. Hence, we can assume $F$ is known, and the following proposition then allows one to reconstruct $G$.

\begin{proposition}
\label{prop:reconstruction}
There is an isomorphism of groups
\begin{equation}
G \to \Aut^{\otimes}(F),
\end{equation}
where $\Aut^{\otimes}(F)$ denotes the group of natural isomorphisms of $F$ which are compatible with the tensor product on $\Rep_{\CC}(G)$.
\end{proposition}
\begin{proof} 
Every element of $G$ acts via a linear map in every finite-dimensional $G$-representation, so define
\begin{equation}
\phi: G \to \Aut^{\otimes}(F): g \mapsto (\rho(g))_{(V,\rho)},
\end{equation}
where $\rho: G \to \GL(V)$ is a $G$-representation. One easily checks that $\phi$ is a well-defined group morphism. Now assume $\phi(g)=\id_F$, then, in particular, 
\begin{equation}
\rho(g)=\id:\Oscr(G) \to \Oscr(G): \delta_h \mapsto \delta_{gh},
\end{equation}
where $(\Oscr(G),\rho)$ is the representation of $G$ on the  algebra of functions on $G$, and $\delta_h$ denotes the indicator function at $h$. This implies that $g=1$, so $\phi$ is injective.

Surjectivity is harder to check. For a given $\alpha \in \Aut^{\otimes}(F)$, we again look at the representation $(\Oscr(G),\rho)$. In fact, one can show that $\alpha_{(\Oscr(G),\rho)}$ is an algebra morphism, and from there one shows that there is a unique $g \in G$ such that\begin{equation}
\alpha_{(\Oscr(G),\rho)}:\Oscr(G) \to \Oscr(G):f \mapsto f(- \cdot g).
\end{equation}
In other words, $\alpha_{(\Oscr(G),\rho)}=\phi(g)_{(\Oscr(G),\rho)}$. This then suffices to ensure that $\alpha=\phi(g)$. We omit the computational details.
\end{proof}

Already from this basic example, it is clear that the  algebra of functions $\Oscr(G)$ of $G$ plays a crucial role.

\subsection{Tannaka-Krein reconstruction}
\label{ss:TK}
From now on we will work over an arbitrary algebraically closed field $\mathbb{K}$. 

The key ingredient in the reconstruction for finite groups was the existence of the forgetful functor $F$ \eqref{eq:forgetful}. In this section we consider the more general setting of a covariant functor $F:\Ascr\r \Vect_{\mathbb{K}}$, where $\Ascr$ is an
\begin{enumerate}
\item essentially small,
\item $\mathbb{K}$-linear,
\item abelian
\end{enumerate} 
category. To this data we will associate a certain coalgebra serving as a substitute for $\Oscr(G)$ which occurs in the proof of Proposition \ref{prop:reconstruction}. We will, however, start by taking a more intuitive dual point of view.

We first give a brief reminder on pseudo-compact algebras, for more details, see \cite[\S 4]{MR3338683}. For us a \textit{pseudo-compact algebra}\index{pseudo-compact algebra} $A$ is a topological algebra whose topology is generated by 2-sided ideals of finite codimension and
which is, moreover, complete. Denote the category of pseudocompact algebras with continuous algebra morphisms by $\PC_{\mathbb{K}}$.

Similarly, a right linear topological $A$-module $M$ is called a \textit{pseudocompact $A$-module}\index{module! pseudo-compact} if its topology is generated by right $A$-submodules of finite codimension, and $M$ is complete. The corresponding category is denoted $\PC(A)$\index{$\PC(A)$}. It is an abelian category.
We will also need the category $\Dis(A)$\index{$\Dis(A)$} of \textit{discrete $A$-modules}\index{module! discrete}, i.e., the right linear topological $A$-modules equipped with the discrete topology. These categories are dual in the following sense. There are functors 
\begin{equation}
\begin{tikzcd}
\Dis(A) \arrow[bend left]{r}{(-)^*} & \PC(A^{\opp}) \arrow[bend left]{l}{(-)^{\circ}}
\end{tikzcd}
\end{equation}
where $(-)^*$\index{$(-)^*$} denotes taking the vector space dual and $(-)^{\circ}$\index{$(-)^{\circ}$} the continuous dual\footnote{For a pseudocompact $A$-module $M$, $M^{\circ}$ consists of the continuous linear functionals $M \to \mathbb{K}$, where $\mathbb{K}$ has the discrete topology.}. These functors define mutually inverse anti-equivalences of categories.

The pseudocompact algebra associated to $(\Ascr,F)$ is denoted $\End(F)$, and consists of all natural transformations $F \Rightarrow F$, with $\mathbb{K}$-linear structure coming from $\Vect_{\mathbb{K}}$, and multiplication defined via composition of natural transformations. The topology is determined by associating to every finite $\alpha \subset \Ob(\Ascr)$ a base open set
\begin{equation}
U(\alpha)=\bigcap_{X \in \alpha} \Ker(\End(F) \to \End(FX)),
\end{equation}
which is an ideal of finite codimension.

It is clear that for every object $X \in \Ascr$, $FX$ is a  finite-dimensional discrete $\End(F)$-module, so we can consider the evaluation functor
\begin{equation}
\label{eq:eval}
\eval_F:\Ascr \to \dist(\End(F)): X \mapsto FX,
\end{equation}
where $\dist(\End(F))$\index{$\dist(\End(F))$} denotes the category of finite-dimensional discrete $\End(F)$-modules. The following theorem is the quintessential \textit{Tannaka-Krein reconstruction theorem}\index{Tannaka-Krein reconstruction theorem}.

\begin{theorem}\cite{MR0232821,MR0472967}
\label{th:tkreconstruction}
If $F$ is faithful and exact, then 
\begin{equation}
\eval_F:\Ascr \to \dist(\End(F))
\end{equation}
defines an equivalence of categories.
\end{theorem}
\begin{proof}
If $\Bscr$ is an essentially small abelian category, then the category $\Ind(\Bscr)$ of \textit{ind-objects}\index{ind-objects} of $\Bscr$ is the category of left exact contravariant functors $\Bscr\r \Ab$.
By \cite[Ch. II]{MR0232821} $\Ind(\Bscr)$ is a Grothendieck category, the Yoneda embedding 
\[
\Bscr\r\Ind(\Bscr):B\mapsto \Bscr(-,B)
\]
is fully faithful and its essential image yields a family of finitely presented generators for $\Ind(\Bscr)$. Moreover,
if every object in $\Bscr$ is noetherian, then the essential image of $\Bscr$ coincides with the category of noetherian objects in $\Ind(\Bscr)$.

We will apply this with $\Bscr=\Ascr^{\opp}$.
Since $F$ is exact and faithful, it follows that $\Ascr$ has finite-dimensional Hom-spaces and every object is of finite length. Therefore $\Ascr^{\opp}$ enjoys the same properties.
Hence, the noetherian objects in the Grothendieck category $\Ind(\Ascr^{\opp})$ are given by $\Ascr^{\opp}$.
Since finite length objects are noetherian, this is then also true for the category of finite length objects.

We claim that $F$ is in fact an injective cogenerator for $\Ind(\Ascr^{\opp})$. We first note that $\Hom(-,F)$
coincides with $F$ when restricted to $\Ascr^{\opp}$. Indeed, the composition is given by
\begin{equation}
\label{eq:restriction}
A\mapsto \Ascr(A,-)\mapsto \Ind(\Ascr^{{\opp}})(\Ascr(A,-),F)=F(A).
\end{equation}
We see in particular that $\Hom(-,F)$ is exact when restricted to $\Ascr^{\opp}$, and therefore $F$ is at least  fp-injective, i.e. $\Ext^1(X,F)=0$ for every finitely presented functor $X$. Since $\Ind(\Ascr^{\opp})$ is locally noetherian, this implies that $F$ is injective (by \cite[A.4, Proposition A.11]{krause2001spectrum}). 
Faithfulness then ensures that $F$ is a
cogenerator. 

By \cite[Ch IV, Prop. 13]{MR0232821} $\End(F)$ is pseudo-compact. Note that the definition in \cite[Ch IV, \S 3]{MR0232821} of pseudo-compactness is more general than ours, but using the fact that $F$ takes
values in finite-dimensional vector spaces, one checks that $\End(F)$ is indeed pseudo-compact in our sense. A similar statement holds for pseudo-compact $\End(F)$-modules.

By \cite[Ch. IV, Th. 4]{MR0232821}, 
 there is a commuting diagram 
\begin{equation}
\label{diag:equivalence}
\begin{tikzcd}
\Ind(\Ascr^{\opp}) \arrow{r}{\Hom(-,F)} &[5em] \PC(\End(F)) \\
\Ascr^{\opp} \arrow{r}{\Hom(-,F)} \arrow[hook]{u}{\text{Yoneda}} & \pc(\End(F)) \arrow[hook]{u}
\end{tikzcd}
\end{equation}
where the horizontal arrows are anti-equivalences of categories, and $\pc(\End(F))$\index{$\pc(\End(F))$} denotes the category of pseudocompact $\End(F)$-modules of finite length. In particular, the lower
row gives an equivalence
\begin{equation}
\label{eq:equivalence}
\Ascr\xrightarrow{\Hom(-,F)} \pc(\End(F))
\end{equation}

Now we observe that we have, in fact, $\pc(\End(F))=\dis(\End(F)$ since the topology on any pseudo-compact $\End(F)$-module is generated by submodules of finite codimension.
Combining this observation with \eqref{eq:restriction} (which shows that $\Hom(-,F)$ restricted to $\Ascr$ is $\operatorname{ev}_F$) we see that \eqref{eq:equivalence} ultimately translates into an equivalence
\[
\operatorname{ev}_F:\Ascr \r \dis(\End(F)).\qedhere
\]
\end{proof}

We now introduce the coalgebra which will play the role of $\Oscr(G)$ for $F:\Ascr \to \Vect_{\mathbb{K}}$. Using the mutually inverse dualities of categories 
\begin{equation}
\begin{tikzcd}
\text{Coalg}_{\mathbb{K}} \arrow[bend left]{r}{(-)^*} & \text{PC}_{\mathbb{K}} \arrow[bend left]{l}{(-)^{\circ}}
\end{tikzcd}
\end{equation}
where $\text{Coalg}_{\mathbb{K}}$ is the category of $\mathbb{K}$-coalgebras, we define 
\begin{equation}
\coend(F):=\End(F)^{\circ}. 
\end{equation}
This gives rise to the equivalence of categories
\begin{equation}
\comod(\coend(F)) \cong \dis(\End(F)),
\end{equation}
where $\comod$ denotes the category of finite-dimensional comodules. 

A more concrete description of $\coend(F)$ can be given as follows: 
\begin{equation}
\label{eq:coend}
\coend(F)=\frac{\bigoplus_{X \in \Ascr} FX^* \otimes FX}{E},
\end{equation}
where $E$ is the following subspace:
\begin{equation}
E=\langle y_* \otimes (Ff)(x) - (Ff)^*(y_*) \otimes x \mid f \in \Hom(X,Y), x \in FX, y_* \in FY^* \rangle_{\mathbb{K}},
\end{equation}
with comultiplication and counit
\begin{align}
\Delta([\xi \otimes x]) &= \sum_i [\xi \otimes x_i] \otimes [\xi_i \otimes x], \\
\epsilon([\xi \otimes x]) &= \sum_i \xi(x_i)\xi_i(x), \end{align}
for $\xi \otimes x \in FX^* \otimes FX$ and $\sum_i \xi_i \otimes x_i \in FX^* \otimes FX$ a dual basis. One checks that $\Delta$ and $\epsilon$ are well-defined and satisfy the coassociativity and counitality conditions. 

\begin{remark}
\label{rem:abstract-coend}
More abstractly, for any functor $G:\Cscr^{\opp} \times \Cscr \to \Dscr$, where $\Cscr$ is small and $\Dscr$ is cocomplete (i.e. it has all small colimits), the $\coend(G)$ can be defined as the colimit
\begin{equation}
\label{eq:colimit}
\begin{tikzcd}
\bigsqcup_{c \to c'} G(c',c) \arrow[shift left=1.2]{r} \arrow[shift right=1.2]{r} & \bigsqcup_{c} G(c,c) \arrow{r} & \coend(G).
\end{tikzcd}
\end{equation}   
In our case, one takes $G:\Ascr^{\opp} \times \Ascr \to \VVect_{\mathbb{K}}:X \mapsto FX^* \otimes FX$, where $\VVect_{\mathbb{K}}$ denotes the category of all $\mathbb{K}$-vector spaces\footnote{The appearance of $\VVect_{\mathbb{K}}$ instead of $\Vect_{\mathbb{K}}$ here is to make sure the colimit \eqref{eq:colimit} makes sense.}.
\end{remark}

From now on, we will only work with $\coend(F)$ and $\comod(\coend(F))$. Theorem \ref{th:tkreconstruction} affords a dictionary between categorical structures on the pair $(\Ascr,F)$ and algebraic structures on the coalgebra $\coend(F)$. An example of how this dictionary works is provided by the following proposition.

\begin{proposition}
\label{prop:extra}
\begin{enumerate}
\item If $\Ascr$ is monoidal, and $F$ is a monoidal functor, then $\coend(F)$ can be made into a bialgebra.
\item If $\Ascr$ moreover has left duals, then $\coend(F)$ can be made into a Hopf algebra.
\item If $\Ascr$ moreover has right duals (i.e. $\Ascr$ is rigid monoidal), then $\coend(F)$ can be made into a Hopf algebra with invertible antipode.
\end{enumerate}
\end{proposition}
\begin{proof}
Consider $(2)$ for example. The antipode is defined as 
\begin{equation}
S:\End(F) \to \End(F):\phi \mapsto S(\phi),
\end{equation}
where $S(\phi)_X=\phi_{X^*}^*:FX \to FX$ for any $X \in \Ascr$. One then checks that $S$ is continuous and the axioms for left duals in $\Ascr$ correspond exactly to the antipode axiom.
\end{proof}

\begin{example}
\label{ex:coordinatering}
If $G$ is an affine algebraic group scheme over $\mathbb{K}$, denote by $\Rep_{\mathbb{K}}(G)$ the rigid monoidal category of finite-dimensional rational $G$-representations and $F:\Rep_{\mathbb{K}}(G) \to \Vect_{\mathbb{K}}$ the forgetful (monoidal) functor. Then one computes that 
\begin{equation}
\coend(F) \cong \Oscr(G)
\end{equation}
as Hopf algebras, and $\eval_F$ is the familiar equivalence between $\Rep_{\mathbb{K}}(G)$ and the category of finite-dimensional comodules over the coordinate ring of $G$.
\end{example}

\section{Presentations of $\uaut(A)$}
\label{sec:pres}
The starting point of our work in \cite{MR3570144,MR3713007} is the following simple observation: in order to construct the bialgebra (resp. Hopf algebra) $\coend(F)$ defined in Section \ref{ss:TK}, one does not need to start from an \textit{abelian} category, not even from a \textit{linear} one. This is only necessary to ensure that the corresponding evaluation functor \eqref{eq:eval} defines an equivalence, as in Theorem \ref{th:tkreconstruction}. For the actual construction of the bialgebra (resp. Hopf algebra) $\coend(F)$, it suffices to start with a monoidal (resp. rigid monoidal category), and these can be defined via generators and relations, much in the same way as monoids and groups. 

In many situations, Hopf algebras appear which are defined abstractly via a universal property (for some examples, see Section \ref{sec:examples}), and we hope to convince the reader that it is often a good idea to try and express them in the form $\coend_{\Cscr}(F)$, for a rigid monoidal category $\Cscr$ defined via generators and relations. This allows for much greater flexibility since there are many ways of changing the pair $(\Cscr,F)$ that do not influence the resulting Hopf algebra. 

\subsection{Algebra presentations for $\coend(F)$}
\label{sec:alg-pres}
We now consider presentations for a (strict) monoidal category $\Cscr$, and use the following notation: 
\begin{equation}
\Cscr=\langle (X_k)_k \hspace{.1cm} | \hspace{.1cm} (\phi_l)_l \hspace{.1cm} | \hspace{.1cm} (\chi_m)_m \rangle_{\otimes}.
\end{equation}
Here, $(X_k)_k$ denote the $\otimes$-generating objects, $(\phi_l)_l$ the $\otimes$-generating morphisms, and $(\chi_m)_m$ the relations among the morphisms. Also, let $F:\Cscr \to \Vect_{\mathbb{K}}$ denote a monoidal functor. Then $\coend(F)$ is a bialgebra by Proposition \ref{prop:extra}, and an algebra presentation can be obtained as follows: 
\begin{enumerate}
\item If $\Ob(\Cscr)$ is not a free monoid on the generating objects $(X_k)_k$, then change the presentation of $\Cscr$ by adding isomorphisms (both arrows and relations) to reduce to this case.
\item Choose bases $(e_{ki})_i$ for each $F(X_k)$.
\item The corresponding ``matrix coefficients" $(z_{kij})_{kij} \in \coend(F)$ are defined via the coaction
\begin{equation}
\label{eq:coaction}
\delta(e_{ki})=\sum_j z_{kij} \otimes e_{kj}.
\end{equation}
They generate $\coend(F)$ as an algebra.
\item Writing out the compatibility of \eqref{eq:coaction} with the generating morphisms $(\phi_l)_l$ produces the relations amongst the generators.
\item The comultiplication and counit are defined via
\begin{align}
\Delta(z_{kij})&=\sum_p z_{kip} \otimes z_{kpj}\\
\epsilon(z_{kij})&=\delta_{ij}
\end{align}
\end{enumerate}
Note that the relations $(\chi_l)_l$ are not used. In practice this process can often be shortened by clever combinatorics. Finally, the procedure outlined above is a more explicit version of formula \eqref{eq:coend} for monoidal categories. In particular, if there are only finitely many $X_k$ and $\phi_l$, then one obtains a finite presentation for $\coend(F)$. In the following sections this will be applied to $\uend(A)$ and $\uaut(A)$ defined in \cite[Ch. 4,7]{MR1016381}.

\subsection{Tannakian reconstruction of $\uend(A)$}
\label{section:tk-end}
Assume $A=TV/(R)$, for $R \subset V \otimes V$ and $\dim_{\mathbb{K}}(V)<\infty$, i.e. $A$ is a quadratic algebra. In \cite[\S 5.3]{MR1016381}, it was shown that the universal bialgebra $\uend(A)$ coacting on $A$ has the following presentation (where $\bullet$ denotes the black product, introduced in \cite[Ch. 3]{MR1016381}):
\begin{equation}
\label{eq:uendblack}
\uend(A)=A^!\bullet A=\frac{T(V^* \otimes V)}{(\sigma_{(23)}(R^{\perp} \otimes R))},
\end{equation}
where 
\begin{equation}
R^{\perp}=\{\phi \in (V \otimes V)^* \mid \phi(R)=0\},
\end{equation} 
and 
\begin{equation}
\sigma_{(23)}:V^* \otimes V^* \otimes V \otimes V \to V^* \otimes V \otimes V^* \otimes V
\end{equation} 
transposes the second and third factors. The structure maps $\Delta$ and $\epsilon$ are given by the usual matrix comultiplication and counit.

Consider the monoidal category
\begin{equation}
\Cscr=\langle r_1,r_2 \mid r_2 \to r_1 \otimes r_1 \rangle_{\otimes}
\end{equation}
and the monoidal functor $F:\Cscr \to \Vect_{\mathbb{K}}$, which is uniquely determined by 
\begin{align}
r_1 &\mapsto V, \\
r_2 &\mapsto R, \\
(r_2 \to r_1 \otimes r_1) &\mapsto (R \hookrightarrow V \otimes V).
\end{align}

\begin{proposition}
There is an isomorphism of bialgebras
\begin{equation}
\coend(F) \cong \uend(A).
\end{equation}
\end{proposition}
\begin{proof}
This follows from implementing the procedure described in Section \ref{sec:alg-pres} or by using formula \eqref{eq:coend}. Let us try the second approach:
\begin{equation}
\label{eq:coend-explicit}
\coend(F)=\frac{\bigoplus_{X \in \Cscr} FX^* \otimes FX}{E}=\frac{T(F(v)^* \otimes F(v))}{E},
\end{equation}
where in the second equality we used that $F(r) \subset F(v) \otimes F(v)$, so we only need to consider sums of powers of $F(v)^* \otimes F(v)$. Now because $F$ is monoidal, the numerator of \eqref{eq:coend-explicit} reduces to $T(V^* \otimes V)$. To compute $E$, it again suffices to consider the generating morphism $r_2 \to r_1 \otimes r_1$, and we see that $E$ is generated by elements
\begin{equation}
\label{eq:coend-end}
y_* \otimes x - y_*|_R \otimes x, \text{ where } y_* \in (V \otimes V)^* \text{ and } x \in R,
\end{equation}
where we identified $x\in R$ with its image under the inclusion $R \hookrightarrow V \otimes V$. This description clearly shows that $E=(\sigma_{(23)}(R^{\perp} \otimes R))$. Finally, one easily checks that the bialgebra structures for $\coend(F)$ and $\uend(A)$ coincide.
\end{proof}

\subsection{Tannakian reconstruction of $\uaut(A)$}
\label{ss:tanrecaut}
At this point we have only provided an alternative construction of $\uend(A)$, but have not gained anything. This situation changes if one considers $\uaut(A)$, as defined in \cite[\S 7.5]{MR1016381}. It was shown there that $\uaut(A)$ can be constructed from $\uend(A)$ by formally adding an infinite number of generators and relations. It is however not clear that this gives rise to a finitely generated Hopf algebra, or even that $\uaut(A)$ does not collapse. 

In order to ensure that $\uaut(A)$ has good properties, we restrict our class of quadratic algebras to Koszul Frobenius algebras, which were considered in \cite[Ch. 8]{MR1016381}. In fact, we will consider their Koszul duals which are more natural from our viewpoint. 

\begin{definition} 
\label{def:AS-regular}
A connected graded algebra $A$ is \textit{Artin-Schelter (AS) regular}\index{Artin-Schelter regular algebra} of dimension $d$ if it has finite global dimension $d$, and
\begin{equation}
\Ext^i_{A}(\mathbb{K},A) = \left \{ 
\begin{array}{ll}
0 & \textrm{if $i \neq d$} \\
\mathbb{K}(l) & \textrm{if $i=d$},
\end{array} \right.
\end{equation}
for some $l \in \mathbb{Z}$ called the \textit{AS-index}\index{AS-index}.
\end{definition}
	
\begin{remark}
This definition requires a few comments:
\begin{enumerate}
\item Many authors also ask for finite GK-dimension (i.e. $A$ is of polynomial growth), but we will not need it in what follows.
\item For AS-regular algebras the AS-index $l>0$, see~\cite[Proposition 3.1]{MR1371143}.
\item One can define both left and right AS-regular algebras, but it turns out that the definition is left-right symmetric, see \cite[Proposition 2.6]{MR3169570}. In particular, Definition \ref{def:AS-regular} is unambiguous.
\end{enumerate}
\end{remark}

\begin{lemma}\cite[Proposition 5.10]{MR1388568} 
\label{lem:koszulfrobas}
Suppose $A$ is Koszul and of finite global dimension, then $A$ is AS-regular if and only if its Koszul dual $A^!$ is Frobenius.
\end{lemma}
\begin{proof}
The Koszul resolution of $A$ looks like 
\begin{equation}
\label{eq:koszul-complex}
K_{\bullet}(A):0 \to A \otimes_{\mathbb{K}} (A^!_d)^* \to \cdots \to A \otimes_{\mathbb{K}} (A^!_1)^* \to A \to \mathbb{K} \to 0,
\end{equation}
which is finite since $A$ has finite global dimension, and in particular $A^! \cong \Ext^{\bullet}_A(\mathbb{K},\mathbb{K})$ is finite-dimensional, with top non-zero degree equal to $d$. 

From Definition \ref{def:AS-regular}, $A$ is AS-regular if and only if the Koszul complex \eqref{eq:koszul-complex} is isomorphic to the complex
\begin{equation}
0 \to A \to A_1^! \otimes A \to \cdots \to A_d^! \otimes A \to 0
\end{equation}
as complexes of right $A$-modules. Using the explicit description of the Koszul differentials, this condition is equivalent to to the existence of an isomorphism of left $A$-modules $A^! \to (A^!)^*[-d]$. The existence of such an isomorphism is in turn equivalent to $A^!$ being Frobenius and hence we are done.
\end{proof}

In \cite[\S 7.5]{MR1016381}, $\uaut(A)$ was introduced as the Hopf envelope of $\uend(A)$, which was denoted $H(\uend(A))$. We will make use of a categorical analogue of this notion.

\begin{proposition}~\cite[Lemma 4.2]{MR1289425}
Let $\Cscr$ be a monoidal category. Then there exists a unique monoidal category $\Cscr^*$ admitting right duals and a monoidal functor $*:\Cscr \to \Cscr^*$ such that for any small monoidal category $\Dscr$ admitting right duals and monoidal functor $F$, there exists a unique monoidal functor $F^*$ making the diagram
\begin{equation}
\begin{tikzcd}
\Cscr \rar{*} \dar[swap]{F} & \Cscr^* \dlar[dashed]{F^*} \\
\mathcal{D}
\end{tikzcd}
\end{equation}
commute. Moreover, this construction is compatible with  Hopf envelopes: for any monoidal  category $\Cscr$ and monoidal functor $F:\Cscr \to \Vect_{\mathbb{K}}$, there are isomorphisms of Hopf algebras
\begin{equation}
H(\coend_{\Cscr}(F)) \cong \coend_{\Cscr^*}(F^*).
\end{equation}
\end{proposition}

\begin{corollary}
\label{cor:aut-coend}
For $(\Cscr,F)$ as in Section \ref{section:tk-end}, there is an isomorphism 
\begin{equation}
\uaut(A) \cong \coend_{\Cscr^*}(F^*)
\end{equation}
of Hopf algebras.
\end{corollary}

Hence, to obtain a presentation for $\uaut(A)$, it suffices to construct \textit{rigidisations}\index{rigidisation} of $(\Cscr,F)$. Because one is only interested in the resulting $\coend$, there is again a lot of extra freedom. The following proposition illustrates the end result of such a construction, and provides a minimal presentation of $\uaut(A)$.

\begin{proposition}\cite[Appendix A]{MR3570144}
\label{prop:pres-aut}
For $A=TV/(R)$ a Koszul, Artin-Schelter regular algebra of global dimension $d$, consider the monoidal category 
\begin{equation}
\Dscr=\langle r_1,r_2, \ldots, r_{d-1},r_d^{\pm 1} \mid r_i \to r_1^{\otimes i}, r_ar_d^{-1}r_{d-a} \to 1 \rangle_{\otimes},
\end{equation}
where $i$ runs over $\{2,\ldots,d\}$ and $a \in \{1, \ldots, d-1\}$ is fixed (so there are $d$ generating morphisms). Define a monoidal functor $G:\Dscr \to \Vect_{\mathbb{K}}$ via
\begin{align}
G(r_1) &= V\\
G(r_i) &= \bigcap_{k+l+2=i}V^{\otimes k} \otimes R \otimes V^{\otimes l} \text{ for } i \geq 2\\
G(r_i \to r_1^{\otimes i}) &= \bigg(\bigcap_{k+l+2=i}V^{\otimes k} \otimes R \otimes V^{\otimes l} \hookrightarrow V^{\otimes i}\bigg)
\end{align}
Then $\coend_{\Dscr}(G) \cong \uaut(A)$.
\end{proposition}

Note that $\Dscr$ only has a finite number of generating objects, so in particular, $\uaut(A)$ is finitely generated. This is a consequence of AS-regularity: it ensures that $\dim(A^!_d)=1$, and that the obvious inclusions
\begin{equation}
(A^!_d)^* \hookrightarrow (A^!_a)^* \otimes (A^!_{d-a})^*
\end{equation}
define non-degenerate pairings between $(A^!_a)^*$ and $(A^!_{d-a})^*$ (cfr. Lemma \ref{lem:koszulfrobas}). This ensures one only has to formally invert a single object in order to construct a pair $(\Dscr,G)$ such that 
\begin{equation}
\coend_{\Dscr}(G)\cong \coend_{\Cscr^*}(F^*) \cong \uaut(A).
\end{equation}
Using Proposition \ref{prop:pres-aut} one can, in fact, check that $\uaut(A)$ is generated by $(z_{ij})_{i,j=1}^d$ (corresponding to the matrix coefficients of $r_1$) and the inverse of a group-like element $\delta$ (corresponding to the matrix coefficient of $r_d$). The relations among these generators are somewhat cumbersome to write down explicitly, so we illustrate them in a special
case.

\begin{example}
\label{ex:two}
Let $A=\mathbb{K}[x,y]$, then 
\begin{equation}
R \subset V \otimes V: r \mapsto x\otimes y-y\otimes x, 
\end{equation}
and $d=2$. Applying the procedure outlined in Section \ref{sec:alg-pres} to $(\Dscr,G)$ from Proposition \ref{prop:pres-aut} we find that $\aut(A)$ is generated, as an algebra, by $a,b,c,d,\delta^{-1}$ with the following relations
\begin{equation}
\label{ref-A.5-161}
\begin{aligned}
ac-ca & =  0 =  bd-db, \\ 
ad-cb & =  \delta  =  da-bc,\\
\delta \delta^{-1} & =  1  =  \delta^{-1}\delta,\\ 
a\delta^{-1}d-b\delta^{-1}c & =  1  =  d\delta^{-1}a-c\delta^{-1}b,\\
b\delta^{-1}a-a\delta^{-1}b & =  0  =  c\delta^{-1}d-d\delta^{-1}c.
\end{aligned}
\end{equation}
Unsurprisingly $\uaut(A)$, like $\uend(A)$, has exponential growth.
\end{example}

\section{Highest weight categories and quasi-hereditary coalgebras}
\label{sec:qh}
The Hopf algebra $\uaut(A)$ satisfies the following universal property, analogous to the one for $\uend(A)$ considered in \cite[Ch. 4,5]{MR1016381}.

\begin{proposition}
\label{prop:universal}
If $H$ is a Hopf algebra and $A$ is an $H$-comodule algebra given by $f:A \to H \otimes A$ such that $f(A_n)\subset H\otimes A_n$, then there is a 
	unique morphism of Hopf algebras $g:\uaut(A) \to H$ such that the diagram 
	$$
	\begin{tikzcd}
	A \drar[swap]{f} \rar{\delta_A} & \uaut(A) \otimes A \dar[dashed]{g \otimes 1} \\
	& H \otimes A
	\end{tikzcd}
	$$
	commutes.
\end{proposition}

For $A=\Sym_{\mathbb{K}}(V)$ a polynomial ring, Proposition \ref{prop:universal} ensures that the coaction
\begin{equation}
A \to \Oscr(\GL(V)) \otimes A
\end{equation}
induced by the standard action of $\GL(V)$ on $V$, factors through $\uaut(A)$. Hence, there is a natural functor 
\begin{equation}
\comod(\Oscr(\GL(V))) \to \comod(\uaut(A)),
\end{equation}
and since (as we saw in Example \ref{ex:coordinatering}) $\comod(\Oscr(\GL(V)))$ is equivalent to $\Rep_{\mathbb{K}}(\GL(V))$, this suggests that $\uaut(A)$ has an interesting representation theory. 

In arbitrary characteristic, $\Rep_{\mathbb{K}}(\Gl(V))$ is an important example of a highest weight category (see Section \ref{sec:highest} for a definition), and the main result of \cite{MR3570144} asserts that $\comod(\uaut(A))$ is also a highest weight category, for any Koszul, Artin-Schelter regular algebra $A$.

\subsection{Representations of the general linear group}
\label{sec:gln}
Consider the category $\Rep_{\mathbb{K}}(G)$ of rational (right) representations\footnote{As before, a representation is by definition finite-dimensional.} of $G=\GL(V)$, for $\dim(V)=n$. By definition, a representation $X$ of $G$ is \textit{rational}\index{representation! rational} if for some (and hence every) basis $e_1, \ldots, e_m$ of $X$, we have
\begin{equation}
\label{eq:pol}
e_i \cdot g=\sum_{j=1}^m e_jf_{ij}(g) \text{ for all } g \in G \text{ and } i=1, \ldots, m,
\end{equation}
for some coefficient functions $f_{ij} \in \Oscr(G)$. A rational representation $X$ can be given the structure of a left comodule by defining the coaction to be
\begin{equation}
\delta:X \to \Oscr(G) \otimes X:e_i \mapsto \sum_{j=1}^m f_{ij} \otimes e_j,
\end{equation}
and this association defines an equivalence of categories between finite-dimensional $\Oscr(G)$-comodules and rational representations of $G$.

The inclusion of monoids $G \hookrightarrow M_n(\mathbb{K})$ induces an inclusion of bialgebras
\begin{equation}
\Oscr(M_n) \hookrightarrow \Oscr(G):\mathbb{K}[x_{11},\ldots,x_{nn}] \hookrightarrow \mathbb{K}[x_{11},\ldots,x_{nn},\determinant^{-1}],
\end{equation}
where $x_{ij}(g)$ is the $ij$-th entry of the matrix $g$ and $\determinant$ is the determinant function. Then a rational representation $X \in \Rep_{\mathbb{K}}(G)$ is called \textit{polynomial}\index{representation! polynomial} if the coefficient functions $f_{ij}$ in \eqref{eq:pol} all belong to $\Oscr(M_n)$. Now for any representation $V \in \Rep_{\mathbb{K}}(G)$, the representation $V \otimes (\wedge^n(V))^{\otimes m}$ is polynomial for some $m \geq 0$, so we can restrict ourselves to studying polynomial representations of $G$, or equivalently, $\Oscr(M_n)$-comodules.

If $\cha(k)=0$, it is a classical fact that $\Rep_{\mathbb{K}}(G)$ is semisimple: every representation in $\Rep_{\mathbb{K}}(G)$ is a direct sum of simple representations. The simple polynomial representations $L(\lambda)$ are classified by the set of partitions with at most $n$ rows
\begin{equation}
\Lambda=\{\lambda=(\lambda_1,\ldots,\lambda_n) \in \NN^n\mid\lambda_1 \geq \lambda_2 \geq \cdots \geq \lambda_n\}.
\end{equation}
This collection can be upgraded to a poset by setting 
\begin{equation}
\lambda \leq \mu \Longleftrightarrow \sum_{j=1}^k \lambda_j \leq \sum_{j=1}^k \mu_j \text{ for all } k=1,\ldots,n.
\end{equation}

Now, consider the subgroups $T \subset B \subset G$ of diagonal and lower triangular matrices. Denoting the simple representation corresponding to $\lambda \in \Lambda$ by $L(\lambda)$, these simple representations can be explicitly constructed as 
\begin{equation}
L(\lambda)=\induced_B^G(\lambda),
\end{equation}
i.e., one considers $\lambda$ as a one-dimensional representation of $T$, extends it to a $B$-representation by letting the unipotent part act trivially, and then induces to a representation of $G$. 

\begin{example}
If $\dim(V)=2$, then for $\lambda=(\lambda_1,\lambda_2) \in \Lambda$, we have
\begin{equation}
L(\lambda)=\induced_B^G(\lambda)=\Sym^{\lambda_1-\lambda_2}(V) \otimes \wedge^2(V)^{\otimes \lambda_2},
\end{equation}
so any simple rational $G$-representation is isomorphic to the tensor product of an integer power of the determinant representation and a symmetric power of the tautological representation.
\end{example}

In positive characteristic the situation is not nearly as easy, as the following simple example shows.

\begin{example}
Assume that $\cha(k)=p>0$, and $V=\mathbb{K}e_1 + \mathbb{K}e_2$. Then 
\begin{equation}
\induced_B^G((p,0))=\Sym^p(V)=\sum_{i=0}^p \mathbb{K}e_1^{p-i}e_2^i
\end{equation}
contains the two-dimensional simple subrepresentation $L=\mathbb{K}e_1^p + \mathbb{K}e_2^p$. Note that $L$ is not even a direct summand of $\Sym^p(V)$.
\end{example}

In particular, we see that $\Rep_{\mathbb{K}}(G)$ is no longer semisimple. It is, however, still true that the simple representations are classified by $\Lambda$, and occur as subrepresentations of the induced representations. For any $\lambda \in \Lambda$, denote $\lambda^*:=-w_0\lambda$, where $w_0$ denotes the longest element in the Weyl group.

\begin{theorem}
For any simple polynomial representation $L$ of $G$, there is a unique $\lambda \in \Lambda$ such that 
\begin{align}
L &\cong \text{soc}(\induced_B^G(\lambda)) \\
& \cong \text{top}(\induced_B^G(\lambda^*)^*).
\end{align}
The simple representation corresponding to $\lambda$ is denoted $L(\lambda)$.
\end{theorem}

This seems like good news, but in fact this information is not even explicit enough to determine the characters of the simple representations. Indeed, to determine these characters is one of the main motivating problems in the field of modular representation theory of reductive algebraic groups.

The theorem does, however, suggest that the representations $\induced_B^G(\lambda)$ and $\induced_B^G(\lambda^*)^*$ still play an important role. For this reason they are denoted $\nabla(\lambda)$ (respectively $\Delta(\lambda)$) and are called \textit{costandard}\index{representation! costandard} (respectively \textit{standard}\index{representation! standard}) representations. We will mostly focus on the costandard representations, since the standard ones are their dual (for a precise statement, see Proposition \ref{prop:dualexc}). Their characters can be computed using the Weyl character formula, and they have the following important properties.

\begin{proposition}
\label{prop:nabla}
For all $\lambda \in \Lambda$:
\begin{enumerate}
\item $\End_{G}(\nabla(\lambda)) \cong \mathbb{K}$,
\item $\Hom_G(\nabla(\lambda),\nabla(\mu)) \neq 0 \Rightarrow \lambda \geq \mu$,
\item $\Ext^1_G(\nabla(\lambda),\nabla(\mu)) \neq 0 \Rightarrow \lambda > \mu$.
\item \label{ext-vanishing}$\Ext^i_G(\Delta(\lambda),\nabla(\mu))=0$ for $i > 0$
\end{enumerate}
\end{proposition}

\begin{corollary}
\label{cor:compcostandard}
Denoting by $[-:-]$ composition multiplicities, we find that $[\nabla(\lambda):L(\lambda)]=1$ and if $[\nabla(\lambda):L(\mu)] \neq 0$, then $\lambda \geq \mu$.
\end{corollary}

This, in turn, indicates that the category of representations filtered by the $\nabla(\lambda)$ plays an important role. Denote by $\Fscr(\nabla)$ the exact subcategory of $\Rep_{\mathbb{K}}(G)$ consisting of representations filtered by the $\nabla(\lambda)$, where $\lambda \in \Lambda$. We similarly define the category $\Fscr(\Delta)$. The following deep result is due to Mathieu.

\begin{theorem}\cite{MR1054234}
\label{th:mat}
If $X,Y \in \Fscr(\nabla)$, then also $X \otimes_{\mathbb{K}} Y \in \Fscr(\nabla)$. The same result holds for $\Fscr(\Delta)$.
\end{theorem}

A last class of representations which we will need, and which play an important role in the representation theory of $\GL(V)$, are the \textit{tilting representations}\index{representation! tilting}. 

\begin{definition}
The category $\Fscr(\nabla) \cap \Fscr(\Delta)$ of all modules having both a $\nabla$-filtration and a $\Delta$-filtration is called the category of tilting modules.
\end{definition}

The name is explained by the following proposition, which follows immediately from Proposition \ref{prop:nabla} \eqref{ext-vanishing}.

\begin{proposition}
For any $M \in \Fscr(\nabla) \cap \Fscr(\Delta)$,
\begin{equation}
\Ext^i_G(M,M)=0, \text{ for all } i>0.
\end{equation}
\end{proposition}

\begin{proposition}
\label{prop:tiltingcomp}
For any indecomposable representation $T \in \Fscr(\nabla) \cap \Fscr(\Delta)$, there is a unique $\lambda \in \Lambda$ and exact sequences
\begin{equation}
0 \to K(\lambda) \to T \to \nabla(\lambda) \to 0,
\end{equation}
and 
\begin{equation}
0 \to \Delta(\lambda) \to T \to K'(\lambda) \to 0
\end{equation}
such that $K(\lambda)$ (respectively $K'(\lambda)$) has a filtration by $\nabla(\mu)$ (respectively $\Delta(\mu)$) with $\mu<\lambda$. This $T$ is denoted $T(\lambda)$\index{$T(\lambda)$}.
\end{proposition}

In particular, we find that 
\begin{equation}
\addc\bigg(\bigoplus_{\lambda \in \Lambda}T(\lambda)\bigg)=\Fscr(\nabla) \cap \Fscr(\Delta),
\end{equation}
where $\addc(M)$\index{$\addc(-)$} denotes the category consisting of all representations isomorphic to direct summands of direct sums of $M$, and this module $T=\oplus_{\lambda \in \Lambda}T(\lambda)$ is called the \textit{characteristic tilting module}\index{module! characteristic tilting} (which is infinite dimensional). 

\begin{example}
For $\dim(V)=2$ and $i<p$, one has 
\begin{equation}
L((i,0))=\nabla((i,0))=\Delta((i,0))=T((i,0))=S^i(V).
\end{equation}
\end{example}

In general, these tilting representations are also hard to describe, but they can be related to more familiar representations. In fact,
\begin{equation}
\wedge^iV=\nabla(\lambda(i))=\Delta(\lambda(i))=L(\lambda(i))=T(\lambda(i))
\end{equation} 
for $i=0, \ldots, n$ and $\lambda(i)=(1, \ldots, 1, 0, \ldots, 0) \in \Lambda$ with $1$ appearing $i$ times. Hence, by Theorem \ref{th:mat} all tensor products of the $\wedge^i(V)$ are tilting representations.

\begin{theorem}
\label{th:tilting}
For $\lambda=(\lambda_1, \ldots, \lambda_n) \in \Lambda$, there is a decomposition
\begin{equation}
\wedge^{\lambda_1^t}V \otimes_{\mathbb{K}} \cdots \otimes_{\mathbb{K}} \wedge^{\lambda_l^t}V=T(\lambda) \oplus Y,
\end{equation}
where $Y$ is a direct sum of tilting representations $T(\mu)$ with $\mu < \lambda$, and moreover $\lambda^t=(\lambda_1^t, \ldots, \lambda_l^t)$ is the conjugate partition.
\end{theorem}

If we denote by 
\begin{equation}
\Vscr=\langle \wedge^iV \mid i=1, \ldots, n \rangle_{\otimes} \subset \Rep_{\mathbb{K}}(G)
\end{equation}
the full monoidal subcategory generated by the exterior powers of the standard representation, and by $F:\Vscr \to \Vect_{\mathbb{K}}$ the restriction of the forgetful functor, then 
\begin{equation}
\coend(F) \cong \Oscr(G).
\end{equation}
Let $\perf(\Vscr^{\opp})$\index{$\perf(\Vscr^{\opp})$} be the triangulated category of finite complexes of finitely generated projective right $\Vscr$-modules. 
Then $\perf(\Vscr^{\opp})$ has a natural structure of monoidal triangulated category by putting 
\begin{equation}
\Vscr(-,u)\otimes \Vscr(-,v)=\Vscr(-,uv)
\end{equation}
and extending to complexes. The functor $F$ extends to an exact monoidal functor
\begin{equation}
\label{ref-7.4-150}
F:\perf(\Vscr^{\opp})\r \derived^b(\Rep_{\mathbb{K}}(G)):\Vscr(-,u)\mapsto F(u).
\end{equation}

\begin{remark}
At the risk of
confusing various tensor products the functor $M$ can be written intrinsically as $-\Lotimes_{\Vscr}M$.
\end{remark}

\begin{theorem}
\label{th:derivedeq}
The functor 
\begin{equation}
	F:\perf(\Vscr^{\opp})\r \derived^b(\Rep_{\mathbb{K}}(G)):\Vscr(-,u)\mapsto F(u).
	\end{equation}
is an equivalence of monoidal triangulated categories.
\end{theorem}
\begin{proof}
This follows from combining Theorem \ref{th:tilting}, Proposition \ref{prop:tiltingcomp} and Corollary \ref{cor:compcostandard}. 
\end{proof}

Note that the tensor generators of $\Vscr$, which govern the derived category of $\Rep_{\mathbb{K}}(G)$ by Theorem \ref{th:derivedeq}, correspond to the terms in the Koszul resolution 
\begin{equation}
0 \to A \otimes \wedge^n V \to \cdots \to A \otimes \wedge^lV \to \cdots \to A \otimes V \to A \to \mathbb{K} \to 0
\end{equation}
of $A=\Sym_{\mathbb{K}}(V)$. We will use this in Section \ref{sec:repsaut} as a starting point to study the representation theory of $\uaut(A)$, for an arbitrary Koszul, Artin-Schelter regular algebra $A$.

\subsection{Highest weight categories and quasi-hereditary coalgebras}
\label{sec:highest}

A lot of the structure present in $\Rep_{\mathbb{K}}(G)$ can be formalised and gives rise to the notion of a \textit{highest weight category}\index{highest weight category}\footnote{This notion plays a very important role in Lie theory, see \cite{bernshtein1971structure,humphreys2008representations}}. It is this notion that we will be able to carry over to the noncommutative setting and $\uaut(A)$. Remember that a poset $(\Lambda,\leq)$ is called \textit{interval finite}\index{poset! interval finite} if for any $\lambda, \mu \in \Lambda$, the set $\{\psi \in \Lambda \mid \lambda \leq \psi \leq \mu\}$ is finite. 

To emphasize the analogy with algebraic groups, whenever dealing with coalgebras, we will use the term \textit{representation} as a synonym for a finite dimensional comodule.

\begin{definition}
\label{def:qh}
Let $C$ be a coalgebra and let $(\Lambda,\leq)$ be an interval finite poset indexing the simple representations. Then $C$ is \textit{quasi-hereditary} if there exist finite-dimensional comodules $\nabla(\lambda)$, for all $\lambda \in \Lambda$, such that:
\begin{enumerate}
\item $\End_C(\nabla(\lambda)) \cong \mathbb{K}$,
\item $\Hom_C(\nabla(\lambda),\nabla(\mu)) \neq 0 \Rightarrow \lambda \geq \mu$,
\item $\Ext^1_C(\nabla(\lambda),\nabla(\mu)) \neq 0 \Rightarrow \lambda > \mu$,
\item $C$ has a filtration with subquotients of the form $\nabla(\lambda)$. 
\end{enumerate}
The category $\comod(C)$ is a highest weight category and the comodules $\nabla(\lambda)$ will be called \textit{costandard comodules}\index{comodule! costandard}.
\end{definition}

\begin{remark}
\begin{enumerate}
\item  The poset $(\Lambda,\leq)$ is a part of the data defining a highest weight category. In particular, there can be different quasi-hereditary structures on the same underlying coalgebra.
\item There is no mention yet of a tensor product. Indeed, for an arbitrary coalgebra, $\comod(C)$ is not necessarily monoidal.
\item Definition \ref{def:qh} is not quite standard and relies on the standardization result of Dlab and Ringel \cite{MR1211481}. See \cite[Appendix A]{MR3713007} for a detailed comparison.
\end{enumerate}
\end{remark}

\begin{example}
\label{ex:directed}
If $C$ is a directed coalgebra, i.e., there is an ordering of the simple representations by a poset $\Lambda$ such that $\Ext^1_C(L(\lambda),L(\mu)) \neq 0 $ implies that $\lambda > \mu$, then one can take $\nabla(\lambda)=L(\lambda)$.
\end{example}

In fact, in some sense a quasi-hereditary coalgebra is a direct generalization of Example \ref{ex:directed}. Indeed, Definition \ref{def:qh} implies that the $\nabla(\lambda)$ form a full exceptional collection in $\derived^b(\comod(C))$. 

\begin{definition}
An object $X$ in a $\mathbb{K}$-linear triangulated category $\Tscr$ is called \textit{exceptional} if 
\begin{equation}
\Hom_{\Tscr}(X,X[r])=
\begin{cases}
\mathbb{K} & \text{ if } r=0,\\
0 & \text{ otherwise}.
\end{cases}
\end{equation}
A collection of objects $\{X_i\}_{i \in (I,\leq)}$ in a $\mathbb{K}$-linear triangulated category $\Tscr$, for some poset $(I,\leq)$, is called an \textit{exceptional collection}\index{exceptional collection} if each $X_i$ is exceptional and:
\begin{equation}
\Hom_{\Tscr}(X_p,X_q[r])=0,
\end{equation}
if $p>q$ and $r \in \mathbb{Z}$. An exceptional collection is called \textit{full} if it generates all of $\Tscr$.
\end{definition}

The full exceptional collection $(\nabla(\lambda))_{\lambda \in \Lambda}$ is special since it consists solely of objects living in the heart of the standard t-structure. 

\begin{definition}
For a full exceptional collection $(F_{\lambda})_{\lambda \in (\Lambda,\leq)}$ in a triangulated category $\Tscr$, there is a dual full exceptional collection $(E_{\lambda})_{\lambda \in (\Lambda,\leq^{\opp})}$, uniquely determined by 
\begin{equation}
\Hom(E_\lambda,F_{\mu}[i])=
\begin{cases}
\mathbb{K}&\text{if $\lambda=\mu$ and $i=0$,}\\
0&\text{otherwise.}
\end{cases}
\end{equation}
\end{definition}

\begin{proposition}
\label{prop:dualexc}
For a highest weight category $\comod(C)$, the exceptional collection dual to the costandard comodules consists again of indecomposable comodules. These comodules are called \textit{standard comodules}\index{comodule! standard}, and are denoted $\Delta(\lambda)$.
\end{proposition}

There is another definition of a
quasi-hereditary 
coalgebra, which is 
often easier to work with in
practice. We will first need the definition of a \textit{heredity chain}\index{heredity chain} which
we will phrase in the context of finite-dimensional algebras. So
assume $A$ is a finite-dimensional $\mathbb{K}$-algebra, with Jacobson
radical~$\rad(A)$\footnote{Recall that the Jacobson radical of $A$ is the intersection of all annihilators of simple right $A$-modules.}.

\begin{definition}
A two-sided ideal~$I$ of~$A$ is called a \textit{heredity ideal}\index{heredity ideal} if
\begin{enumerate}
\item $I$ is idempotent,
\item $I_{A}$ is projective,
\item $I \rad(A) I=0$.
\end{enumerate} 
\end{definition}

\begin{definition}
The algebra $A$ is a quasi-hereditary algebra if it has a filtration by heredity ideals, i.e., there is a chain
\begin{equation}
0=J_0 \subset J_1 \subset \cdots \subset J_{m-1} \subset J_m=A
\end{equation}
of ideals of~$A$ such that for any~$1 \leq t \leq m$,~$J_t/J_{t-1}$ is a heredity ideal in $A/J_{t-1}$. Such a chain is called a heredity chain.
\end{definition}

\begin{definition}
\label{def:qh-ring}
A (possibly infinite-dimensional) coalgebra $C$ is \textit{quasi-hereditary}\index{coalgebra! quasi-hereditary} if there exists an exhaustive filtration
\begin{equation}		
0 \subset C_1 \subset C_2 \subset \cdots \subset C_n \subset \cdots
\end{equation}
of finite-dimensional subcoalgebras such that for every $i$, we have
\begin{equation}
0=(C_i/C_i)^* \subset (C_i/C_{i-1})^* \subset (C_i/C_{i-2})^* \subset \cdots \subset C_i^*
\end{equation}
is a heredity chain. Such a filtration is called a \textit{heredity cochain}\index{heredity cochain}.
\end{definition}

\section{Representations of $\uaut(A)$}
\label{sec:repsaut}
Consider a Koszul, Artin-Schelter regular algebra $A=TV/(R)$ of global dimension $d$, and the corresponding universal Hopf algebra $\uaut(A)$. Taking our cue from Theorem \ref{th:derivedeq} and the preceding constructions, consider the Koszul resolution of $A$: 
\begin{equation}
0 \to A \otimes R_d \to \cdots \to A \otimes R_l \to \cdots \to A \otimes R \to A \otimes V \to A \to \mathbb{K} \to 0,
\end{equation}
with\footnote{We usually omit tensor product signs.} $
R_l:=\bigcap_{i+j+2=l}V^i R V^j $. In particular, we have $R_2=R$ and for uniformity we also put $R_1=V$. It follows from the basic properties of AS-regular algebras that $\dim R_d=1$ and that, moreover, the obvious inclusions $R_d\hookrightarrow R_a R_{d-a}$ define non-degenerate pairings between $R_a$ and $R_{d-a}$. These properties characterise the AS-regular algebras among the Koszul ones, as we saw in Lemma \ref{lem:koszulfrobas}.

Since we would like to think of $\uaut(A)$ as a noncommutative version of the coordinate ring of $\GL(V)$, we will denote $\comod(\uaut(A))$ by $\Rep_{\mathbb{K}}(\uaut(A))$\index{$\Rep_{\mathbb{K}}(\uaut(A))$}. It is easy to see that the $(R_l)_l$ are $\uaut(A)$-comodules, with $R_d$ being invertible. The discussion after Theorem \ref{th:derivedeq} suggests to consider 
\begin{equation}
\Vscr=\langle R_l \mid l=1, \ldots, d \rangle_{\otimes} \subset \Rep_{\mathbb{K}}(\uaut(A)).
\end{equation}
Moreover, denoting $F:\Vscr \to \Vect_{\mathbb{K}}$ the restriction of the forgetful functor, one might \emph{expect} in analogy with the commutative setting that 
\begin{equation}
\coend(F) \cong \uaut(A),
\end{equation}
and that $F$ induces an equivalence
\begin{equation}
\label{eq:perf2}
F:\perf(\Vscr^{\opp}) \to \derived^b(\Rep_{\mathbb{K}}(\uaut(A))):\Vscr(-,v) \mapsto F(v)
\end{equation}
of monoidal triangulated categories. However it seems difficult to verify this directly since the structure of $\Rep_{\mathbb{K}}(\uaut(A)))$ is completely unknown at this stage.
Therefore we proceed differently.

We relate $\Vscr$ to a certain monoidal category $\Uscr$ with strong combinatorial features. In fact, we have already introduced a suitable monoidal category
$\Dscr$ in Proposition \ref{prop:pres-aut} but the latter was optimized for finding a compact presentation of $\uaut(A)$. In contrast, $\Uscr$ more faithfully reflects the representation-theoretic features of $\uaut(A)$.

To summarize: we will not use $(\Dscr,G)$ like in Proposition \ref{prop:pres-aut} but use a different pair $(\Uscr,M)$.  
Nonetheless, we have
\begin{equation}
\coend_{\Uscr}(M)\cong\coend_{\Dscr}(G)\cong\uaut(A).
\end{equation} 
This illustrates the fact that different pairs $(\Cscr,F)$ with the same $\coend$ can be used to study different aspects of the same Hopf algebra.
\subsection{The category $\Uscr$} 
\label{sec:comb-structure}
It is not hard to see that there are morphisms of $\uaut(A)$-representations
\begin{align}
\Phi_{a,b}&:R_{a+b} \to R_a R_b, \\
\Theta_{a,b}&:R_a R_d^{-1} R_b \to R_{a+b-d},
\end{align}
satisfying certain natural relations. To formalize this, define the monoid
\begin{equation}
\Lambda=\langle r_1,\ldots, r_{d-1},r_d^{\pm 1} \rangle
\end{equation} 
and consider the following monoidal categories with set of objects $\Lambda$:
\begin{equation}
\Uscr_{\u}=\langle r_1, \ldots, r_{d-1}, r_d^{\pm 1} \mid \phi_{a,b}:r_{a+b} \to r_ar_b \rangle_{\otimes},
\end{equation}
and impose the following set of relations:
\begin{equation}
\label{rel:1}
\begin{tikzcd}[column sep=large]
r_{a+b+c}\ar{d}{\phi_{a+b,c}}\ar{r}{\phi_{a,b+c}} & r_ar_{b+c}\ar{d}{r_a\phi_{b,c}}\\
r_{a+b}r_{c}\ar{r}{\phi_{a,b}r_c} & r_ar_{b}r_c
\end{tikzcd}
\end{equation}
writing $u$ for $\Id_u$ and suppressing tensor products as usual. Similarly, consider 
\begin{equation}
\Uscr_{\d}=\langle r_1, \ldots, r_{d-1},r_d^{\pm 1} \mid \theta_{a,b}:r_ar_d^{-1}r_b \to r_{a+b-d} \rangle_{\otimes}
\end{equation}
and impose the relations:
\begin{equation}
\label{rel:2}
\begin{tikzcd}[column sep=large]
r_ar^{-1}_dr_br_d^{-1}r_c\ar{r}{\theta_{a,b}r_d^{-1}r_c} \ar{d}{r_ar_d^{-1} \theta_{b,c}}& r_{a+b-d}r_d^{-1}r_c\ar{d}{\theta_{a+b-d,c}}\\
r_{a} r_d^{-1} r_{b+c-d}\ar{r}{\theta_{a,b+c-d}}& r_{a+b+c-2d}
\end{tikzcd}
\end{equation}
Relations \eqref{rel:1} and \eqref{rel:2} are chosen because they are satisfied by the morphisms $\Phi_{a,b}$ and $\Theta_{a,b}$.

Now let $\tilde{\Uscr}=\Uscr_{\d}\ast \Uscr_{\u}$ be the category with set of objects $\Lambda$ and the morphisms 
freely generated by the morphisms in $\Uscr_{\d}$ and $\Uscr_{\u}$. Then $\tilde{\Uscr}$ is strict monoidal
in the obvious way.
Let $\Uscr$ be the
monoidal quotient of $\tilde{\Uscr}$ obtained by imposing the following sets of relations
	
	\begin{enumerate}
	\item
	\begin{equation}
	\label{rel:3}
	\begin{tikzcd}[column sep=large]
	r_{a+b}r_d^{-1}r_c\ar{r}{\phi_{a,b}r_d^{-1}r_c}\ar{d}{\theta_{a+b,c}}&r_{a}r_b r_d^{-1} r_c\ar{d}{r_a\theta_{b,c}}\\
	r_{a+b+c-d}\ar{r}{\phi_{a,b+c-d}}&r_{a} r_{b+c-d}
	\end{tikzcd}
	\end{equation}
	where $d\le b+c$ and where moreover we allow the degenerate cases $a+b=d$ in which case
	we put $\theta_{d,c}=\Id_{r_c}$ and $b+c=d$ in which case we put $\phi_{a,0}=\Id_{r_a}$.
	\item
	\begin{equation}
	\label{rel:4}
	\begin{tikzcd}[column sep=large]
	r_ar_{d}^{-1} r_{b+c}\ar{r}{r_ar_d^{-1}\phi_{b,c}}\ar{d}{\theta_{a,b+c}}& r_a r_d^{-1} r_br_c\ar{d}{\theta_{a,b}r_c}\\
	r_{a+b+c-d}\ar{r}{\phi_{a+b-d,c}}& r_{a+b-d}r_c
	\end{tikzcd}
	\end{equation}
	where $d\le a+b$ and where, moreover, we allow the degenerate cases $b+c=d$ when we put
	$\theta_{a,d}=\Id_{r_a}$, and $a+b=d$ when we put $\phi_{0,c}=\Id_{r_c}$.
	\end{enumerate}

Again, it is easy to see that the $\Phi_{a,b}$ and $\Theta_{a,b}$ satisfy relations \eqref{rel:3} and \eqref{rel:4}.

By linearising the morphism spaces in these monoidal categories, we obtain linear categories $\mathbb{K}\Uscr_\u, \mathbb{K}\Uscr_\d$ and $\mathbb{K}\Uscr$. It is possible to grade the morphisms of these categories, so one can consider them as multiple object versions of graded algebras. The combinatorial structure of $\Uscr$ is elucidated in the following proposition.

\begin{proposition}\cite[Propositions 3.1.2, 3.3.1]{MR3570144}
\label{prop:comb}
\begin{enumerate}
\item The graded categories $\mathbb{K}\Uscr_\u$ and $\mathbb{K}\Uscr_{\d}$ are Koszul,
\item $\Uscr$ can be given the structure of a \textit{Reedy category}\index{Reedy category}, i.e., every morphism $f$ in $\Uscr$ can be written uniquely as a composition $f_{\u}\circ f_{\d}$ with $f_{\d}$ in $\Uscr_{\d}$ and $f_{\u}$ in $\Uscr_{\u}$.
\end{enumerate}
\end{proposition}

\subsection{$\uaut(A)$ is quasi-hereditary}
\label{sec:autA-qh}
Since relations \eqref{rel:1}, \eqref{rel:2}, \eqref{rel:3} and \eqref{rel:4} were chosen based on the relations satisfied by the $\Phi_{a,b}$ and $\Theta_{a,b}$, it follows that by construction there is a monoidal functor $G: \Uscr \to \Vscr$, which can be composed with the forgetful functor $F$ to obtain a monoidal functor
\begin{equation}
\label{eq:monfunc}
M: \Uscr \to \Vect_{\mathbb{K}}.
\end{equation}
Note that up to some morphisms and relations, this setup is very similar to Proposition \ref{prop:pres-aut}, so the following theorem should not come as a surprise.

\begin{theorem}\cite[Theorem 5.1]{MR3570144}
The monoidal category $\Uscr$ is rigid, and there is an isomorphism of Hopf algebras 
\begin{equation}
\coend_{\Uscr}(M) \cong \uaut(A).
\end{equation}
\end{theorem}

At this point we forget about the intermediate category $\Vscr$, which is a priori hard to control since it is a linear category, and focus only on $\Uscr$, which is not linear, and on the functor $M$. This turns out to greatly simplify calculations.

Equipping $\Lambda$ with the left- and right-invariant partial ordering generated by 
\begin{align}
r_{a+b} &<  r_ar_b, \\
r_{d-a-b} &< r_{d-a}r_d^{-1}r_{d-b}, 
\end{align}
we can now state the main theorem from \cite{MR3570144} more precisely.

\begin{theorem}
\label{th:quasiher}
The coalgebra $\uaut(A)$ is quasi-hereditary with respect to the poset $(\Lambda,\leq)$. The costandard and standard representations are given as 
\begin{equation}
\label{ref-1.3-8} 
\begin{aligned}
	\nabla(\lambda) &= \coker \bigg( \bigoplus_{\substack{\mu \to \lambda \text{ in }\Uscr\\ \mu < \lambda}} M(\mu) \to M(\lambda) \bigg)\\
	\Delta(\lambda) &= \ker \bigg( M(\lambda) \to \bigoplus_{\substack{\lambda \to \mu \text{ in } \Uscr \\ \mu < \lambda}} M(\mu) \bigg).
\end{aligned}
\end{equation}
\end{theorem}

The proof of this theorem uses the strong combinatorial structure on the category $\Uscr$ from Proposition \ref{prop:comb} in order to check Definition \ref{def:qh-ring}. Assuming that $\Lambda_1\subset\Lambda_2$ are saturated subsets\footnote{A subset $\pi \subset \Lambda$ is called \textit{saturated}\index{poset! saturated} if $\mu \leq \lambda \in \pi$ implies $\mu \in \pi$.} of $\Lambda$ such that the elements of $\Lambda_2-\Lambda_1$ are incomparable. Let $\Uscr_i\subset \Uscr$ be the full subcategories of $\Uscr$ with object sets $\Lambda_i$. 
The key technical result that enters in the proof of Theorem \ref{th:quasiher} is the following.

\begin{theorem}
\label{th:heredity}
There is an exact sequence
\begin{equation}
\label{ref-1.4-11}
	0 \to \prod_{\lambda\in\Lambda_2-\Lambda_1}\Hom_{\mathbb{K}}(\nabla(\lambda),\Delta(\lambda)) \to \End_{\Uscr_2}(M) \to \End_{\Uscr_1}(M) \to 0,
\end{equation}
 where $\nabla(\lambda)$ and $\Delta(\lambda)$ are as in \eqref{ref-1.3-8} above.
\end{theorem}

Starting with \eqref{ref-1.4-11} we may construct
a heredity cochain in $\coend_{\Uscr}(M)$ which yields that $\coend_{\Uscr}(M)\cong \uaut(A)$ is quasi-hereditary. The following theorem provides an analogue of Theorem \ref{th:mat}.

\begin{corollary}
\label{cor:filtrations}
Denote by $\Fscr(\Delta)$ (respectively $\Fscr(\nabla)$) the categories of $\uaut(A)$-comodules that have a $\Delta$-filtration (respectively $\nabla$-filtration). Then:
\begin{enumerate}
\item $\Fscr(\Delta)$ and $\Fscr(\nabla)$ are closed under tensor products. 
\item \label{cor:Mlambda}$M(\lambda)\in \Fscr(\Delta)\cap \Fscr(\nabla)$.
\end{enumerate}
\end{corollary}

So we see that the $M(\lambda)$ indeed play a role analogous to the tensor products of the $\wedge^iV$ for $\GL(V)$. In fact, it turns out that the linearisation of $\Uscr$ 
is equivalent to $\Vscr$, and we even obtain an analogue of Theorem \ref{th:derivedeq}.
\begin{theorem}
\label{th:mainth}
The monoidal functor 
\begin{equation}
M:\mathbb{K}\Uscr\r \Rep_{\mathbb{K}}(\uaut(A)):\lambda\mapsto M(\lambda)
\end{equation}
is  fully faithful and its essential image is $\Vscr$. Moreover, the derived version of $M$
\begin{equation}
\label{eq:derived}
M:\perf(\Uscr^{\opp})\r \derived^b(\Rep_{\mathbb{K}}(\uaut(A)))
\end{equation}
induced by $\mathbb{K}\Uscr(-,\lambda)\mapsto M(\lambda)$ is an equivalence of monoidal triangulated categories.
\end{theorem}
\begin{corollary}
The representation ring of $\uaut(A)$ is given by 
\begin{equation}
\ZZ\langle r_1, \ldots, r_{d-1},r_d,r_d^{-1} \rangle
\end{equation}
where $r_i$ corresponds to $[R_i]$.
\end{corollary}
\begin{proof}
In a quasi-hereditary coalgebra, the costandard representations $[\nabla(\lambda)]$ form a $\ZZ$-basis of the representation ring $G_0(\uaut(A))$. By Corollary \ref{cor:filtrations}\eqref{cor:Mlambda}, $M(\lambda) \in \Fscr(\nabla)$, and one can show that they are related to the costandard representations by a unitriangular matrix. Hence, the $[M(\lambda)]$ also form a $\ZZ$-basis of $G_0(\uaut(A))$. Since $M$ is monoidal and maps $r_i$ to $M(r_i)=R_i$, we are done.
\end{proof}

This gives yet more motivation for thinking of $\uaut(A)$ as a noncommutative version of the coordinate ring of $\GL(V)$, since the representation ring of $\GL_d$ is of the form $\ZZ[r_1,\ldots,r_{d-1},r_d,r_d^{-1}]$.

\subsection{Co-Morita equivalences}

Two bialgebras are said to be \textit{co-Morita equivalent} if their monoidal categories of comodules are equivalent, as monoidal $\mathbb{K}$-linear categories. In Theorem \ref{th:mainth}, the domain category does not depend on the specific $A$ we started with, but only on its global dimension $d$. In fact, the equivalence \eqref{eq:derived} may be used to transfer the standard $t$-structure on the derived category $\derived^b(\Rep_{\mathbb{K}}(\uaut(A)))$ to one on $\perf(\Uscr^{\opp})$. This can be used to give an intrinsic description of the induced $t$-structure referring solely to  properties of $\Uscr$. As a corollary we obtain:

\begin{theorem}
\label{comorita}
The category $\Rep_{\mathbb{K}}(\uaut(A))$ as a monoidal category only depends on the global dimension of $A$ and not on $A$ itself.
\end{theorem}

In other words, by letting $A$ vary we obtain non-trivial examples of co-Morita equivalent\index{co-Morita equivalence} Hopf algebras~\cite[\S5]{MR1408508}. 

This somewhat curious corollary to Theorem \ref{th:mainth} turns out to be a special case of a much more general phenomenon. Given two monoidal functors $F,G:\Cscr \to \Vect_{\mathbb{K}}$ on a fixed rigid monoidal category $\Cscr$, one can form an algebra $\cohom(F,G)$, by using a slight variation of Remark \ref{rem:abstract-coend}.

\begin{theorem}\cite{thesis}
\label{th:dualhopf}
If $\cohom(F,G) \neq 0$, then cotensoring with it induces an equivalence 
\begin{equation}
\Rep_{\mathbb{K}}(\coend(G)) \to \Rep_{\mathbb{K}}(\coend(F))
\end{equation}
of monoidal categories.
\end{theorem}

\begin{remark}
\label{rem:cohom}
The relation $\cohom(F,G)\neq 0$ is actually an equivalence relation and divides
up the monoidal functors $\Cscr\r \Vect_{\mathbb{K}}$ into connected components whose members yield co-Morita equivalent Hopf algebras.
\end{remark}

\begin{proof}[Proof of Theorem \ref{comorita} using Theorem \ref{th:dualhopf}]
Consider two Koszul AS-regular algebras $A$ and $B$, both of global dimension $d$. Section \ref{sec:autA-qh} provides us with two monoidal functors $M_A,M_B:\Uscr \to \Vect_{\mathbb{K}}$, such that 
\begin{align}
\coend(M_A)&\cong \uaut(A), \\
\coend(M_B)&\cong \uaut(B).
\end{align}
By Theorem \ref{th:dualhopf}, it suffices to show that $\cohom_{\Uscr}(M_A,M_B) \neq 0$. 
Consider the (saturated) subset $\{1\} \subset \Lambda$, and the corresponding one-object category $\mathbf{1} \subset \Uscr$. By a suitable analogue of Theorem \ref{th:heredity} one shows that  
\begin{equation}
\cohom_{\bf{1}}(M_{A}|_{\bf{1}},M_{B}|_{\bf{1}})\hookrightarrow\cohom_{\Uscr}(M_A,M_B).
\end{equation}
Since $\mathbb{K}\cong \cohom_{\bf{1}}(M_{A}|_{\bf{1}},M_{B}|_{\bf{1}})\neq 0$, we are done.
\end{proof}

\section{Representations of $\uend(A)$}
\label{sec:repsend}
One might wonder whether the relation between $\uend(A)$ and $\uaut(A)$ is as close as the relation between $\Oscr(M_n)$ and $\Oscr(\GL(V))$ (introduced in Section \ref{sec:gln}). That this is indeed the case follows from the following proposition. Denote by $\Lambda^+=\langle r_1,\ldots,r_d \rangle \subset \Lambda$.
\begin{proposition}
\label{prop:subalgebra}
The bialgebra $\uend(A)$ is the minimal subcoalgebra of $\uaut(A)$ whose representations
have simple composition factors belonging to the set $\{L(\lambda)_{\lambda\in \Lambda^+}\}$.
\end{proposition}

Proposition \ref{prop:subalgebra} allows us to recover most of the results from \cite{MR3299599}, which describe the representation theory of $\uend(A)$, for $A$ Koszul. For the sake of exposition, we will restate the main result of \cite{MR3299599} in the language of Section \ref{sec:repsaut}, and in the case where $A$ is also AS-regular. To do this, we need to consider the monoidal subcategory $\Uscr^+_\u$ of $\Uscr_{\u}$, generated by the objects $r_1, \ldots, r_d$.

\begin{theorem}\cite[Theorem 4.3, Corollary 4.4]{MR3299599}
\label{th:kvdb}
For a Koszul AS-regular algebra $A$ of global dimension $d$, there is a monoidal functor 
$$
M^+:\mod(\Uscr_{\u}^{+,\opp})\r \Rep_{\mathbb{K}}(\uend(A)):\Uscr^+_{\u}(-,u)\mapsto M^+(u)
$$
which is an equivalence of monoidal categories.
\end{theorem}

In fact, our techniques ensure that $\uend(A)$ is quasi-hereditary with $\Delta(\lambda)=M(\lambda)$. Moreover, the $(\Delta(\lambda) )_{\lambda\in \Lambda^+}$ form a system of projective generators for 
$\Rep_{\mathbb{K}}(\uend(A))$ and we obtain the equivalence.

The representation theory of $\uend(A)$ can be understood in terms of quivers with relations as follows. For every number $n \geq 0$, set $C_n=\mathbb{K}Q_n/I$ to be the quiver with relations corresponding to an $n$-dimensional (directed) hypercube $Q_n$ with commuting faces. Depending on the global dimension $d$, we need to consider certain full subalgebras $C_{n,d} \subset C_n$, obtained by deleting some of the vertices of the hypercubes, and then there is an equivalence of (abelian) categories
\begin{equation}
\Rep_{\mathbb{K}}(\uend(A)) \cong \mod(\Uscr_{\u}^{+,\opp}) \cong \bigoplus_n \mod(C_{n,d}).
\end{equation}
Rather than spelling out the somewhat contrived (though easy to implement in practice) combinatorial rule for constructing $C_{n,d}$, we refer the reader to \cite{MR3299599}.

\section{Representations of $\uaut(\mathbb{K}[x,y])$}
\label{sec:repstwo}
Theorem \ref{comorita} tells us that to understand the representation theory of $\uaut(A)$, with $A$ of global dimension $d$, it suffices to study the representations of $\uaut(\Sym_{\mathbb{K}}(V))$, with $\dim_{\mathbb{K}}(V)=d$, and the functor 
\begin{equation}
\Rep_{\mathbb{K}}(\uaut(\Sym_{\mathbb{K}}(V))) \to \Rep_{\mathbb{K}}(\uaut(A))
\end{equation}
that realises the equivalence of monoidal categories. 

Since $\Sym_{\mathbb{K}}(V)$ is a commutative ring, one might hope that there is an even closer connection between representations of $\uaut(\Sym_{\mathbb{K}}(V))$ and $\Rep_{\mathbb{K}}(\GL(V))$. For this reason, we will denote $\uaut(\Sym_{\mathbb{K}}(V))$ by $\Oscr_{\tt nc}({{\GL_d}})$, which we think of as as some sort of noncommutative coordinate ring of $\GL(V)$.

For $V$ of dimension $1$, we find 
\begin{equation}
\Oscr_{\tt nc}({{\GL_1}})=\Oscr(T)=\mathbb{K}[t,t^{-1}],
\end{equation}
the coordinate ring of a $1$-dimensional torus, so $d=2$ is the first interesting case. In this section we review the results from \cite{MR3713007}, where this example is treated in detail. 

\subsection{A noncommutative version of the Borel and torus subgroups}

In Section \ref{sec:gln}, we saw the importance of the torus and Borel subgroups in the representation theory of $\GL(V)$. For $\Oscr_{\tt nc}(\GL_2)$ it is possible to define analogues of the coordinate rings of these subgroups, using the explicit presentation from Example \ref{ex:two}:
\begin{align}
\Oscr_{\tt nc}(B)&=\Oscr_{\tt nc}({{\GL_2}})/(b) \cong \mathbb{K} \langle c,d^{\pm 1} \rangle [a^{\pm 1}], \\
\Oscr(T)&=\Oscr_{\tt nc}({{\GL_2}})/(b,c) \cong  \mathbb{K}[a^{\pm 1},d^{\pm 1}].
\end{align}
Here $\Oscr(T)$ is the (commutative) coordinate ring of a two-dimensional torus~$T$. We identify
its character group $X(T)$  with the Laurent monomials in $a,d$. By sending $r_2\in \Lambda$ to $ad \in X(T)$ and $r_1\in \Lambda$ to $d\in X(T)$ we obtain a map of monoids ${\tt wt}:\Lambda\rightarrow X(T)$\index{${\tt wt}$}. We can now easily imitate the construction of the induced representations. 

If $t\in X(T)$, then there is an associated $1$-dimensional $\Oscr(T)$-representation $\mathbb{K}_t$
which may also be viewed as an $\Oscr_{\tt nc}(B)$-representation. Denote by $\Ind^{{\GL_2}}_B $ the
right adjoint to the restriction functor 
\begin{equation}
\Res^{\GL_2}_B:\Rep_{\mathbb{K}}(\Oscr_{\tt nc}({{\GL_2}}))\rightarrow \Rep_{\mathbb{K}}(\Oscr_{\tt nc}(B)).
\end{equation}
Then we have the following result: 

\begin{theorem} 
There is a decomposition
\[
\induced_B^{{\GL_2}}(\mathbb{K}_t) = \bigoplus_{\substack{\lambda \in \Lambda \\ {\tt wt}(\lambda)=t}} \nabla(\lambda).
\]
\end{theorem}

In particular, we see that $\induced^{{\GL_2}}_B(\mathbb{K}_t)=0$ if $t\not\in X(T)^+:=\im {\tt wt}$. This agrees with  the commutative case where only dominant weights yield non-zero representations under induction. But we also
see that, in contrast to the commutative case, the induced representations are not indecomposable here. However,
they still yield all costandard comodules. 

\subsection{The simple representations}

From the fact that  $\Oscr_{\tt nc}({{\GL_2}})$ is quasi-hereditary it follows by the general theory of these algebras that the simple comodules are of the form $L(\lambda)=\im(\Delta(\lambda)\rightarrow \nabla(\lambda))$;
this, in principle, reduces their study
to a linear algebra problem. 

This  problem is usually difficult to solve, but in this particular case it is possible. The bialgebra $\uend(A)$ is the subalgebra (by Proposition \ref{prop:subalgebra}) of $\Oscr_{\tt nc}({{\GL_2}})$ generated by $a,b,c,d$
and we have:

\begin{theorem}
Assume that $\cha(\mathbb{K})=0$. All simple $\Oscr_{\tt nc}(\GL_2)$-representations are repeated tensor products of simple $\uend(\mathbb{K}[x,y])$-representations and their duals.
\end{theorem}

The simple $\uend(\mathbb{K}[x,y])$-representations can be understood using Theorem \ref{th:kvdb}, and were considered in \cite{MR3299599}. They are tensor products of $(\Sym^n(V))_{n \in \NN}$ and $\wedge^2V$, where $V$ denotes the standard representation. Thus every simple $\Oscr_{\tt nc}(\GL_2)$-representation is a tensor product of these basic representations and their duals. It is somewhat intricate to characterize which among those tensor products are simple, but this is achieved in \cite{MR3713007}.

\section{More examples of universal Hopf algebras}
\label{sec:examples}

Various other types of universal Hopf algebras  have  also been considered in the literature, see for example \cite{MR1812048,MR2302731,MR3029329,chirvasitu2016quantum,MR3275027,MR3545505}. Most of them arise as quotients of $\uaut(A)$, for some algebra $A$. In this section we will discuss one such example  and we will also give some more comments on the commutative case
which was touched upon in Theorem \ref{th:derivedeq}.

\subsection{The universal quantum group of a non-degenerate bilinear form}
In \cite{MR1068703}, Dubois-Violette and Launer introduced the \textit{universal quantum group of a non-degenerate bilinear form}\index{universal quantum group of a non-degenerate bilinear form}. Their definition is equivalent to the following.
\begin{definition}
Given a vector space $V$ such that $1<\dim(V)<\infty$, and a non-degenerate bilinear form $b:V \otimes V \to \mathbb{K}$, the universal quantum group $H(b)$\index{$H(b)$} of $b$ is the universal Hopf algebra coacting on $V$ making $b$ into an $H(b)$-comodule morphism (for the trivial comodule structure on $\mathbb{K}$).
\end{definition}

\begin{example}
For $q \in \mathbb{K}^{\times}$, consider the bilinear form given by 
$$
b=\begin{pmatrix} 0 & 1 \\ -q^{-1} & 0 \end{pmatrix}.
$$
One then computes that $H(b)=\Oscr_q(\SL_2)$, the quantum coordinate ring of $\SL_2$ (see, for example, \cite{MR1321145}).
\end{example}

Setting $A=TV^*/(b)$, we obtain a Koszul AS-regular algebra, see \cite[Proposition 1.1]{zhang1998non}, and one verifies that there is an isomorphism of Hopf algebras 
\begin{equation}
\label{iso:slhb}
\usl(A) \cong H(b),
\end{equation}
where $\usl(A)$ was introduced in \cite[\S 8.5]{MR1016381}. 

\medskip

There is a close connection between $H(b)$ and the famous \textit{Temperley-Lieb category}\index{Temperley-Lieb category}. 
The Temperley-Lieb category is best known for its attractive graphical model, based on planar non-intersecting strands (see \cite{abramsky2007temperley}),
but abstractly it can be characterised by the following presentation
\begin{equation}
\label{eq:temperley}
\begin{aligned}
\Uscr=\langle v | 1 \xrightarrow{\phi} vv, vv \xrightarrow{\psi} 1 |& v \xrightarrow{1 \otimes \phi} vvv \xrightarrow{\psi \otimes 1} v=\id_v=v \xrightarrow{\phi \otimes 1} vvv \xrightarrow{1 \otimes \psi} v, \\
&v \xrightarrow{1 \otimes \phi} vvv \xrightarrow{1 \otimes \psi} v=v \xrightarrow{\phi \otimes 1} vvv \xrightarrow{\psi \otimes 1} v \rangle_{\otimes}.
\end{aligned}
\end{equation} 
The generators $\psi$ and $\phi$ correspond to cups and caps in the graphical model.
The third relation ensures that the ``circle'' $\eta:=\psi\phi\in \End(1)$  acts centrally on $\Uscr$.
One has 
\begin{equation}
H(b)=\coend(F_b),
\end{equation} 
where $F_b:\Uscr\r \Vect_{\mathbb{K}}$ is the monoidal functor
with $F_b(v)=V$, $F_b(\psi)=b$ and $F_b(\phi):\mathbb{K}\r V\otimes V$ being dual to $b$.

The image $q(b):=F_b(\eta)\in \End(\mathbb{K})=\mathbb{K}$  of $\eta$ under $F_b$ is called the \emph{quantum dimension}\index{quantum dimension} of $V$ (it can be zero!). 
The quantum dimension divides the space of monoidal functors $\Uscr \to \Vect_{\mathbb{K}}$ into connected components.

\begin{theorem}\cite[Theorem 1.1]{MR1998031}
\label{bichonmorita}
Let $b,b'$ be two non-degenerate bilinear forms. Then there is a monoidal equivalence
$$
\Rep_{\mathbb{K}}(H(b)) \cong \Rep_{\mathbb{K}}(H(b'))
$$
if and only if $q(b)=q(b')$.
In particular, if we choose $q\in \mathbb{K}^\times$ such that $q+q^{-1}=q(b)$, then there is an equivalence of monoidal categories 
$$
\Rep_{\mathbb{K}}(H(b)) \cong \Rep_{\mathbb{K}}(\Oscr_q(\SL_2)).
$$
\end{theorem}

Theorem \ref{bichonmorita} is in stark contrast with Theorem \ref{comorita}, despite the fact that \eqref{iso:slhb} seems to indicate that $\uaut(A)$ and $H(b)$ are quite close.

\subsection{Spiders and representations of $\GL(V)$}
We briefly indicate how the techniques from \S\ref{sec:repsaut} fit in well with the planar diagrammatic approach to the representation theory of algebraic groups and Lie algebras. 

Recall that in Theorem \ref{th:derivedeq} we showed that 
\begin{equation}
\Vscr=\langle \wedge^iV \mid i=1, \ldots, n \rangle_{\otimes} \subset \Rep_{\mathbb{K}}(G),
\end{equation}
for $G=\GL(V)$ and $\dim(V)=n$, gives rise to a derived equivalence
\begin{equation}
	F:\perf(\Vscr^{\opp})\r \derived^b(\Rep_{\mathbb{K}}(G)):\Vscr(-,u)\mapsto F(u).
\end{equation}
Motivated by Theorem \ref{th:mainth}, one might wonder if the category $\Vscr$ has a nice combinatorial presentation as monoidal category. 

We will assume that $\cha(\mathbb{K})=0$\footnote{In this case there is no need for derived categories: closing $\Vscr$ under direct summands recovers $\Rep_{\mathbb{K}}(G)$ since representations of $G$ are completely reducible.}, though see Remark \ref{rem:poschar}. As generating morphisms one can take the natural maps
\begin{align}
\Gamma_{i,j}&:\wedge^iV \otimes \wedge^jV \to \wedge^{i+j}V\\
\Sigma_{i,j}&:\wedge^{i+j}V \to \wedge^iV \otimes \wedge^jV.
\end{align}
The relations among the $\Gamma_{i,j}$ and $\Sigma_{i,j}$ can be determined using \textit{skew Howe duality}\index{skew Howe duality}, allowing for a planar diagrammatic description of the category $\Vscr$. A morphism between tensor products of the $\wedge^iV$ is encoded as a certain kind of oriented graph, resulting in a category called the \textit{$\GL(V)$-spider}\index{$\GL(V)$-spider}. See \cite[Theorem 3.3.1]{MR3263166} for more details on skew Howe duality and spiders.

\begin{remark}
\label{rem:poschar}
For $\cha(\mathbb{K})>0$, there is a version of skew Howe duality by Adamovich and Rybnikov~\cite{MR1390747}, but the resulting algebraic presentation is quite involved. The diagrammatic approach seems to be more flexible, see \cite{elias}. R.~Howe formulated what is now called ``Howe duality" from a unifying point of view of Lie superalgebras. For a lucid exposition of Howe duality and its further development, see \cite{leites2002howe}.
\end{remark}

\bibliography{biblio}
\bibliographystyle{amsplain}

\end{document}